\documentclass{amsart}
\usepackage[all]{xy}
\usepackage{verbatim}
\usepackage{graphicx}
\usepackage{color}
\usepackage{amsthm}
\usepackage{amssymb}
\usepackage[colorlinks=true]{hyperref}
\usepackage{footmisc}

\usepackage{tikz}

\numberwithin{equation}{section}

\newtheorem{theorem}[equation]{Theorem}
\newtheorem*{theorem*}{Theorem} \newtheorem{lemma}[equation]{Lemma}

\newtheorem*{conjecture*}{Mamma Conjecture}
\newtheorem*{conjecture1*}{Mamma Conjecture (revisited)}
\newtheorem{proposition}[equation]{Proposition}
\newtheorem{corollary}[equation]{Corollary}
\newtheorem*{corollary*}{Corollary}

\theoremstyle{remark}
\newtheorem{definition}[equation]{Definition}

\newtheorem{notation}[equation]{Notation}

\theoremstyle{remark}
\newtheorem{remark}[equation]{Remark}

\newcommand{\cD}{{\mathcal D}}

\newcommand{\cH}{{\mathcal H}}
\newcommand{\cI}{{\mathcal I}}

\newcommand{\cR}{{\mathcal R}}

\newcommand{\cT}{{\mathcal T}}

\newcommand{\bbA}{\mathbb{A}}

\newcommand{\bbC}{\mathbb{C}}

\newcommand{\bbL}{\mathbb{L}}
\newcommand{\bbN}{\mathbb{N}}
\newcommand{\bbP}{\mathbb{P}}
\newcommand{\bbR}{\mathbb{R}}

\newcommand{\bbQ}{\mathbb{Q}}
\newcommand{\bbZ}{\mathbb{Z}}

\newcommand{\too}{\longrightarrow}

\newcommand{\ie}{\textsl{i.e.}\ }
\newcommand{\eg}{\textsl{e.g.}}

\let\oldmarginpar\marginpar
\def\marginpar#1{\oldmarginpar{\tiny #1}}

\begin{document}

\title[Feynman quadrics-motive of the massive sunset graph]{Feynman quadrics-motive \\of the massive 
sunset graph}
\author{Matilde Marcolli and Gon{\c c}alo~Tabuada}

\address{Division of Physics, Mathematics, and Astronomy, California Institute of Technology, Pasadena, CA 91125, USA}
\email{matilde@caltech.edu}
\urladdr{http://www.its.caltech.edu/~matilde}
\address{Gon{\c c}alo Tabuada, Department of Mathematics, MIT, Cambridge, MA 02139, USA}
\email{tabuada@math.mit.edu}
\urladdr{http://math.mit.edu/~tabuada}
\thanks{Matilde Marcolli was supported by the NSF grant DMS-1707882 and Gon{\c c}alo Tabuada by the NSF CAREER Award \#1350472}

\date{\today}

\abstract{We prove that the Feynman quadrics-motive of the massive sunset graph is ``generically'' not mixed-Tate. Moreover, we explicitly describe its ``extra'' complexity in terms of a Prym variety.}}



\maketitle
\vskip-\baselineskip
\vskip-\baselineskip


\section{Introduction}\label{sec:intro}
\subsection*{Feynman motive}
After the seminal work of Bloch-Esnault-Kreimer \cite{BEK}, there has been a lot of research concerned with the construction of motives associated to Feynman graphs $\Gamma$; consult
\cite{AluMa2, AluMa3, BloKr, BloKr2, BoBr, BrSch, BrSch2, BrSchY, CeyMar, Mar, Mar2, MarNi, Sch1, Sch2}. Several different approaches have been developed for the construction of
such ``Feynman motives". In all the cases the main problem is the construction of a motive
$M_\Gamma$ such that the (renormalized) Feynman integral of $\Gamma$ is
a period of $M_\Gamma$. There are different possible ways of writing the Feynman integral
(\eg\ Feynman parametric form, momentum space and configuration space) and each one of these ways leads to a different Feynman motive. The most commonly studied approach is the Feynman parametric form. In this case, given a base field $F$, the Feynman motive $M_\Gamma$ is defined as the Voevodsky's mixed motive $M({\mathbb P}^{n-1}\backslash X_\Gamma)_\bbQ \in \mathrm{DM}_{\mathrm{gm}}(F)_\bbQ$, where $n$ stands for the number of internal edges of 
$\Gamma$ and $X_\Gamma$ for the graph hypersurface defined by the vanishing of the Kirchhoff--Symanzik  
polynomial of $\Gamma$; consult \cite{BeBro, BEK, BrSch, Sch1} for details. When renormalization is taken into account, ${\mathbb P}^{n-1}\backslash X_\Gamma$
needs to be replaced by a certain blow-up of itself; see \cite{BEK, BloKr}. Currently, one of the most important open questions concerning Feynman graphs is the following:

\vspace{0.1cm}

{\em Question: Given a Feynman graph $\Gamma$, is the associated Feynman motive $M_\Gamma$ mixed-Tate? If not, how to describe its ``extra'' complexity?}  

\vspace{0.1cm} 

On the ``positive side'', $M_\Gamma$ is known to be mixed-Tate whenever $\Gamma$ has less than $14$ edges; see \cite{Sta,Ste}. Consult also \cite{AluMa4} for infinite families of mixed-Tate Feynman motives. On the ``negative side'', there exist examples of Feynman graphs $\Gamma$ with $14$ edges for which the Feynman motive $M_\Gamma$ is {\em not} mixed-Tate; see \cite{Doryn, Sch1}.

All the above can be generalized to the case where the Feynman graph $\Gamma$ is equipped with a mass parameter $m$. In this generality, the computation of the associated Feynman motive $M_{(\Gamma, m)}$ becomes much more difficult and only a few computations are currently known; consult \cite{AluMa,BloHov2, BloHov1, Vanhove}.

\subsection*{Feynman quadrics-motive}
As observed by Bloch-Esnault-Kreimer in \cite[\S5]{BEK} (see also \cite[\S1]{Mar}), the Feynman parametric form is not the only way of writing the Feynman integral of $\Gamma$. An alternative approach is to write the Feynman integral of $\Gamma$ in terms of edge propagators. This alternative approach is developed in \S\ref{sub:quadrics-motive}, where we also consider the case where $\Gamma$ is equipped with a mass parameter $m$. As explained in {\em loc. cit.}, given a spacetime dimension $D$,  this approach leads to the {\em Feynman quadrics-motive} $M^Q_{(\Gamma, m)}:=M(\bbP^{b_1(\Gamma) D}\backslash Q_{(\Gamma, m)})_\bbQ \in \mathrm{DM}_{\mathrm{gm}}(F)_\bbQ$, where $b_1(\Gamma)$ stands for the first Betti number of $\Gamma$ and $Q_{(\Gamma, m)}$ for the union $\bigcup_{i=1}^n Q_{i, \epsilon}$ of certain ``deformed'' quadric hypersurfaces; consult Definition \ref{FQmotdef}. Intuitively speaking, the Feynman motive and the Feynman quadrics-motive are two different ``motivic incarnations'' of the same period (= Feynman integral of $(\Gamma, m)$); see \S\ref{sub:Regularization}.

\subsection*{Statement of results}
Assume that the base field $F \subseteq \bbC$ is algebraically closed. 
Consider the {\em massive sunset graph}\footnote{$l_1$ and $l_2$ are called the associated ``loop variables''; consult \S\ref{sub:quadrics-motive} for details.} with mass parameter $m=(m_1, m_2, m_3) \in \bbQ^3$:
\begin{center}
\includegraphics[scale=0.25]{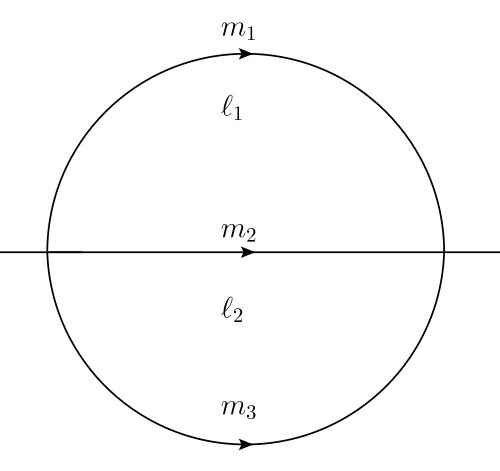}
\end{center}
In this case, the ``deformed'' quadric hypersurfaces $Q_{1, \epsilon}$, $Q_{2, \epsilon}$, and $Q_{3, \epsilon}$, corresponding to the $3$ internal edges, are odd-dimensional. Hence, following Beauville \cite[\S6.2]{Beau}, whenever $Q_{1, \epsilon} \cap Q_{2, \epsilon} \cap Q_{3, \epsilon}$ is a complete intersection, we have an associated (abelian) Prym variety $\mathrm{Prym}(\widetilde{C}/C)$, where $C$ stands for the discriminant divisor of the quadric fibration associated to the triple intersection and $\widetilde{C}$ for the \'etale double cover of the curve $C$.

Our main result answers the above important question, with the Feynman motive replaced by the Feynman quadrics-motive, in the case of the massive sunset graph:

\begin{theorem}\label{thm:main-new}
Let $(\Gamma,m)$ be the massive sunset graph. Assume that the spacetime dimension $D$ is $\geq 2$ and that the mass parameter $m=(m_1, m_2, m_3)$ satisfies $m_3^2 \neq m_1^2 + m_2^2$. Under these assumptions, the following holds:
\begin{itemize}
\item[(i)] The Feynman quadrics-motive $M^Q_{(\Gamma, m)}$ is {\em not} mixed-Tate.
\item[(ii)] The Feynman quadrics-motive $M^Q_{(\Gamma, m)}$ belongs to the smallest subcategory of $\mathrm{DM}_{\mathrm{gm}}(F)_\bbQ$ which can be obtained from the following set of motives 
$$\{M(\mathrm{Prym}(\widetilde{C}/C))_\bbQ, \bbQ(-1), \bbQ(1)\}$$ 
by taking direct sums, shifts, summands, tensor products, and at most $5$ cones.
\end{itemize}
\end{theorem}

Roughly speaking, Theorem \ref{thm:main-new} shows that the Feynman quadrics-motive associated to the massive sunset graph is not mixed-Tate for a ``generic'' choice of the mass parameter $m$. Moreover, Theorem \ref{thm:main-new} provides an explicit ``upper bound'' for the complexity of the Feynman quadrics-motive. In particular, the ``obstruction'' to the mixed-Tate property is explicitly realized by the Prym variety $\mathrm{Prym}(\widetilde{C}/C)$.
\begin{remark}[Related work]
Let $(\Gamma, m)$ be the massive sunset graph. In the particular case of equal masses, Bloch-Vanhove proved in \cite[Lem.~6.1]{BloHov1} that the Feynman motive $M_{(\Gamma, m)}$ is also not mixed-Tate. Their approach is based on an explicit description of the parametric Feynman integral in terms of elliptic curves. In contrast, our approach to prove that the Feynman quadrics-motive $M^Q_{(\Gamma, m)}\not\simeq M_{(\Gamma, m)}$ is not mixed-Tate is based on an explicit description of the Chow motives of complete intersections of two and three quadric hypersurfaces; consult \S\ref{sec:proof} for details.
\end{remark}

\section{Feynman quadrics-motive}\label{Qsec}
Let $D>0$ be the spacetime dimension and $(\Gamma, m, \kappa)$ a Feynman graph equipped with a mass parameter $m$ and with external momentum $\kappa$. Recall that $\Gamma$ is a finite, connected, and direct graph. In what follows, we will write $E(\Gamma)$, $E_{\mathrm{ext}}(\Gamma)$ and $E_{\mathrm{int}}(\Gamma)$ for the set of edges, external edges, and internal edges, respectively. Similarly, we will write $V(\Gamma)$ and $V_{\mathrm{int}}(\Gamma)$ for the set of vertices and internal vertices, respectively. Given an edge $e \in E(\Gamma)$, we will denote by $s(e)$, $t(e)$ and $\partial(e)$ its source, target, and boundary, respectively. The {\em mass parameter} $m=(m_i)$ consists of a rational number $m_i \in \bbQ$ indexed by the internal edges $e_i \in E_{\mathrm{int}}(\Gamma)$. In the same vein, the {\em external momentum} $\kappa = (\kappa_j)$ consists of a vector $\kappa_j=(\kappa_{j, r}) \in \bbQ^D$ indexed by the external edges $e_j \in E_{\mathrm{ext}}(\Gamma)$.

\subsection{Feynman integral}\label{sub:Feynman}
To every internal edge $e_i\in E_{\mathrm{int}}(\Gamma)$ associate a ``momentum variable'' $k_i=(k_{i,r})\in \bbA^D$ and the following {\em edge propagator}:
\begin{equation}\label{equadric}
 q_i(k_i) = \sum_{r=1}^D k_{i,r}^2 + m_i^2\,.
\end{equation} 
Under these notations, recall from \cite[\S3.1]{Mar} that the (unrenormalized) {\em Feynman integral} $\cI_{(\Gamma, m, \kappa)}$ associated to the triple $(\Gamma, m, \kappa)$ is defined as follows: 
\begin{equation}\label{IGamma}
C \int \frac{\prod_{v\in V_{\mathrm{int}}(\Gamma)} \delta(\sum_{e_i\in E_{\mathrm{int}}(\Gamma)} \epsilon_{v,i} k_i + \sum_{e_j\in E_{\mathrm{ext}}(\Gamma)} \epsilon_{v,j} \kappa_j)}{\prod_{e_i\in E_{\mathrm{int}}(\Gamma)} q_i(k_i)} \prod_{e_i\in E_{\mathrm{int}}(\Gamma)} \frac{d^D k_i}{(2\pi)^D}\,.
\end{equation}
Some explanations are in order: $C$ stands for the product $\prod_v \lambda_v (2\pi)^{-D}$ with $\lambda_v$ the coupling constant at the vertex $v$; $\epsilon_{v,i}$ for the incidence matrix with entries $1$, $-1$, or $0$, according to whether $v=s(e)$, $v=t(e)$, or $v\notin \partial(e)$, respectively (similarly for $\epsilon_{v,j}$); $\prod_{e_i} d^D k_i$ for the standard volume form in $\bbA^{nD}(\bbR)$ with $n:=\#E_{\mathrm{int}}(\Gamma)$ the number of internal edges of $\Gamma$; and finally $\delta$ stands for the delta function.
\subsection{Quadrics-motive}\label{sub:quadrics-motive} Let us denote by $n:=\#E_{\mathrm{int}}(\Gamma)$ the number of internal edges of $\Gamma$. In what follows, we will always assume that the mass parameter $m$ is {\em positive}, \ie that $m_i >0$ for every internal edge $e_i \in E_{\mathrm{int}}(\Gamma)$.
\begin{notation}\label{not:pairing}
Given any two vectors $v=(v_{i,r}) \in \bbA^{nD}$ and $v'=(v'_{i,r}) \in \bbA^{nD}$, let $\langle v, v' \rangle :=\sum_{i=1}^n\sum_{r=1}^D v_{i,r} v'_{ir}$ and $v^2 :=\langle v, v \rangle =\sum_{i,r} v_{i,r}^2$.
\end{notation}
Note that due to the presence of the mass parameter $m$, the polynomial \eqref{equadric} in $D$ variables is non-homogeneous. By (formally) adding an homogeneous coordinate $x$, we can consider the associated homogeneous polynomial in $D+1$ variables:
\begin{equation}\label{qex}
 q'_i(k_i,x) := \sum_{r=1}^D k_{i,r}^2 + m_i^2x^2\,.
\end{equation} 
Moreover, under the identification of $k_i=(k_{i,r}) \in \bbA^D$ with the vector $v=(v_{j,r})$ of $\bbA^{nD}$ defined as 
$k_{i,r}$ if $i=j$ and $0$ otherwise, the polynomial \eqref{qex} can be considered as an homogeneous polynomial in $nD+1$ variables (where $k=(k_i) \in \bbA^{nD}$):
\begin{equation}\label{qex1}
 q'_i(k,x) := k_i^2+ m_i^2x^2\,.
\end{equation} 
Let us denote by $Q'_i \subset \bbP^{nD}$ the associated quadric hypersurfaces.

The delta function $\delta$ in the numerator of \eqref{IGamma} imposes linear relations between the ``momentum variables'' $k_i = (k_{i,r})\in \bbA^D$. Concretely, every internal vertice $v \in V_{\mathrm{int}}(\Gamma)$ yields the following linear relations:
\begin{equation}\label{linrels}
\sum_{\substack{e_i \in E_{\mathrm{int}}(\Gamma)\\ s(e_i)=v}} k_i + \sum_{\substack{e_j \in E_{\mathrm{ext}}(\Gamma) \\ s(e_j)=v}} \kappa_j= \sum_{\substack{e_i \in E_{\mathrm{int}}(\Gamma) \\ t(e_i)=v}} k_i + \sum_{\substack{e_j \in E_{\mathrm{ext}}(\Gamma)\\  t(j)=v}} \kappa_j .
\end{equation}
Let us write $N$ for the number of {\em independent} linear relations imposed by \eqref{linrels}, and choose $n-N$ {\em independent} variables $l_i$ among $\{k_1, \ldots, k_n\}$. One usually refers to the variables $l_i$ as the ``loop variables''. Indeed, it is known that $N=\# V_{\mathrm{int}}(\Gamma) -1$. Therefore, the difference $n-N=\# E_{\mathrm{int}}(\Gamma)-\# V_{\mathrm{int}}(\Gamma) +1$ is equal to the first Betti number $b_1(\Gamma)$ of the graph $\Gamma$; see \cite[\S5.2]{CoMa}\cite[\S8]{ItZu}. In what follows, we will write $L:=b_1(\Gamma)$ for the ``loop number" of $\Gamma$.
\begin{lemma}\label{FQlem}
Assume that $\kappa=0$. Under this assumption, the intersection of the quadric hypersurfaces $Q'_i \subset \bbP^{nD}$ with the following linear subspace
$$ H_\Gamma := \bigcap_{v\in V_{\mathrm{int}}(\Gamma)} \{ \sum_{\substack{e_i \in E_{\mathrm{int}}(\Gamma)\\ s(e_i)=v}} k_i - \sum_{\substack{e_i \in E_{\mathrm{int}}(\Gamma) \\ t(e_i)=v}} k_i =0 \} \subset \bbP^{nD}$$
determines quadric hypersurfaces $Q_i\subset \bbP^{LD}$, $i=1,\ldots, n$.
\end{lemma}

\proof  When $\kappa=0$, the relations \eqref{linrels} become homogeneous in the variables $k_i$
and determine the linear subspace $H_\Gamma$. Note that the loop variables $\ell_i$
and the auxiliary variable $x$ are homogeneous coordinates of $H_\Gamma$. Let us then write 
$q_i(\ell,x)$ for the restriction of \eqref{qex1} to $H_\Gamma$. Since $H_\Gamma \simeq \bbP^{LD}$, the intersections $Q_i:= Q'_i\cap H_\Gamma$ agree with the quadric hypersurfaces in $\bbP^{LD}$
defined by the equations $\{q_i(\ell,x)=0\}$.
\endproof

\begin{notation}\label{uvardef}
Let $u=(u_0: \cdots : u_{LD})$ be projective coordinates on the projective space $\bbP^{LD}$ corresponding to the following variables:
\begin{eqnarray}\label{uvars}
u_0:=x & (u_1, \ldots, u_D):=\ell_1 \quad \cdots & (u_{(L-1)D}, \ldots, u_{LD}):=\ell_{L}\,.
\end{eqnarray}
\end{notation}
Note that the parameterizing space of all quadric hypersurfaces in $\bbP^{LD}$
is the projective space $\bbP^{\binom{LD+2}{2} -1}$ of 
symmetric $(LD+1)\times (LD+1)$-matrices up to scalar multiples.
Inside this parameterizing space we have the discriminant hypersurface $\cD$ consisting of all those quadratic
forms with non-trivial kernel. Recall that a {\em net of $n$ quadric hypersurfaces} in $\bbP^{LD}$ consists of an embedding $\rho\colon \bbP^{n-1} \hookrightarrow \bbP^{\binom{LD+2}{2} -1}$. 

Consider the following net of quadrics
\begin{eqnarray}\label{eq:net1}
\rho\colon \bbP^{n-1} \hookrightarrow \bbP^{\binom{LD+2}{2} -1} && (0: \cdots : 0: \underset{i}{1} : 0 : \cdots: 0) \mapsto Q_i\,,
\end{eqnarray}
where $Q_i$ stands for the quadric hypersurface of the above Lemma \ref{FQlem}.

\begin{lemma}\label{FQ3}
The above net of quadrics \eqref{eq:net1} has the following properties:
\begin{itemize}
\item[(i)] The quadric hypersurfaces $Q_i$ belong to $\bbP^{\binom{LD+2}{2} -1}(\bbR)$, \ie the defining quadratic form $q_i$ of the quadric $Q_i$ is real.
\item[(ii)]  The symmetric matrices $A_i$, defined by the equalities $q_i(u)=\langle u, A_i u\rangle$, can be written as $A_i =T^\dagger_i T_i$, where $T_i^\dagger$ stands for the adjoint of the matrix $T_i$ with respect to the bilinear form of Notation~\ref{not:pairing}.
\item[(iii)] Let $P$ be the projection $(u_0,\ldots,u_{LD})\mapsto (u_1,\ldots, u_{LD})$ and
$\bar T_i$ the matrix $PT_i P$. Under these notations,  the matrices $\bar T_i$ satisfy the following momentum conservation condition: $\sum_{s(e_i)=v} \bar T_i = \sum_{t(e_i)=v} \bar T_i$.
\end{itemize}
\end{lemma} 
\begin{proof}
Item (i) follows from the combination of Lemma~\ref{FQlem} with the description \eqref{qex1} of the quadratic form $q'_i$. The description \eqref{qex1} of the quadratic forms $q'_i$ also implies that the matrix $A_i$ can be written as $A_i = T_i^\dagger T_i$ with $q_i(u)=\langle T_i u, T_i u \rangle$. This shows item (ii). In what concerns item (iii), note that thanks to Lemma~\ref{FQlem} we can order the internal edges of the Feynman graph $\Gamma$ in such a way that the first $L$ quadratic forms $q_i$ do {\em not} depend on the variables $l_j \in \bbA^D$ with $j \neq i$. The dependence of the remaining quadratic forms $q_i$ on the $l_i$ variables is dictated by the momentum conservation condition \eqref{linrels} (with $\kappa_j=0$). Note that the matrices $\bar T_i$ are diagonal and that its entries are either $1$ or $0$. This corresponds to which variables $u_j, j \geq 1$, occur in $q_i$ or not. Therefore, the condition \eqref{linrels} (with $\kappa_j=0$) can be written as $ \sum_{s(e_i)=v} \bar T_i u = \sum_{t(e_i)=v} \bar T_i u$. Finally, since the latter equality holds for all the variables $u_j$, it can be re-written as the following momentum conservation condition: $\sum_{s(e_i)=v} \bar T_i = \sum_{t(e_i)=v} \bar T_i$. This proves item (iii).
\end{proof}

\begin{definition}
A {\em one-parameter deformation} of a net of $n$ quadric hypersurfaces $\rho:\bbP^{n-1} \hookrightarrow \bbP^{\binom{LD+2}{2} -1}$ is
a morphism $\tilde\rho: \bbP^{n-1}\times \bbA^1 
\to \bbP^{\binom{LD+2}{2} -1}$ such that $\rho=\tilde\rho |_{\bbP^{n-1}\times \{ 0 \}}$. Given a point $\epsilon \in \bbA^1(\bbQ)$, with $\epsilon\neq 0$, we will write $\rho_\epsilon$ for the associated net of quadrics $\tilde\rho |_{\bbP^{n-1}\times \{ \epsilon \}}$ and call it the {\em $\epsilon$-deformation of $\rho$}.
\end{definition}
\begin{proposition}\label{3quaddef}
There exists a one-parameter deformation $\tilde\rho$ of the net of quadrics \eqref{eq:net1} such that for sufficiently small points $\epsilon \in \bbA^1(\bbQ)$ the $\epsilon$-deformations
\begin{eqnarray}\label{eq:deformation}
\rho_\epsilon\colon \bbP^{n-1} \hookrightarrow \bbP^{\binom{L
D+2}{2} -1} && 
(0: \cdots : 0: \underset{i}{1} : 0 : \cdots: 0) \mapsto Q_{i, \epsilon}
\end{eqnarray}
have the following properties:
\begin{itemize}
\item[(i)] The quadrics $Q_{i,\epsilon}$ belong to $\bbP^{\binom{LD+2}{2} -1}\backslash \cD$, \ie they are smooth.
\item[(ii)] The quadrics $Q_{i,\epsilon}$ belong to $\bbP^{\binom{L D+2}{2} -1}(\bbR)$.
\item[(iii)] The symmetric matrices $A_{i,\epsilon}$ can be written as $A_{i,\epsilon} =T^\dagger_{i,\epsilon} T_{i,\epsilon}$.
\item[(iv)] Let $P$ be the projection $(u_0,\ldots,u_{LD})\mapsto (u_1,\ldots, u_{LD})$ and 
$\bar T_{i,\epsilon}$ the matrix $PT_{i,\epsilon} P$. Under these notations, the matrices $\bar T_{i,\epsilon}$ satisfy the following momentum conservation condition: $\sum_{s(e_i)=v} \bar T_{i,\epsilon} = \sum_{t(e_i)=v} \bar T_{i,\epsilon}$.
\end{itemize}
\end{proposition}
\begin{corollary}\label{noreal}
The quadrics $Q_{i,\epsilon}$ don't have real points.
\end{corollary} 
\begin{proof} 
Items (ii)-(iii) of Proposition \ref{3quaddef} imply that the symmetric matrices $A_{i,\epsilon}$ have real non-negative eigenvalues $\lambda_{i,\epsilon}$. Moreover, item (i) implies that these eigenvalues $\lambda_{i,\epsilon}$ are strictly positive. Therefore, we have $q_{i,\epsilon}(u) =\sum_{i=0}^{LD} \lambda_{i,\epsilon} u_i^2$ with $\lambda_{i,\epsilon}>0$. This shows that the associated quadrics $Q_{i, \epsilon}$ don't have real points.
\end{proof}
\begin{proof}{(of Proposition \ref{3quaddef})} Thanks to Lemma~\ref{FQ3}, the quadrics $Q_i$ of Lemma \ref{FQlem} belong to $\bbP^{\binom{L D+2}{2} -1}(\bbR)$. However, they are singular in general, \ie they belong to the discriminant divisor $\cD$. Nevertheless, since the complement $\bbP^{\binom{LD+2}{2} -1}\backslash \cD$ is a Zariski open set, there exists then a one-parameter deformation $\tilde\rho$ of the net of quadrics \eqref{eq:net1} such that for a generic point $\epsilon \in \bbA^1(\bbQ)$ the associated $\epsilon$-deformation \eqref{eq:deformation} satisfies conditions (i)-(ii). The corresponding quadratic forms can then be diagonalized $q_{i,\epsilon}(u) =\sum_i  \lambda_{i,\epsilon} u_i^2$. Moreover, the eigenvalues $\lambda_{i,\epsilon}$ are real, non-zero, and converge to the eigenvalues $\lambda_i$ of the quadratic forms $q_i$ when $\epsilon \to 0$. If $\lambda_i>0$, then for a sufficiently small point $\epsilon \in \bbA^1(\bbQ)$ we also have $\lambda_{i,\epsilon}>0$. If $\lambda_i=0$, then for a sufficiently small point $\epsilon \in \bbA^1(\bbQ)$, we have $\lambda_{i,\epsilon}>0$ or $\lambda_{i,\epsilon}<0$. In the latter case, we can always change the sign in order to make all the eigenvalues positive. This yields a new $\epsilon$-deformation which not only satisfies conditions (i)-(ii) but also condition (iii).

Let us now prove condition (iv). Recall that a {\em spanning tree $\tau$} of a connected graph $\Gamma$ is a 
connected subgraph of $\Gamma$ which is a tree and which contains all the vertices of $\Gamma$. The Euler characteristic formula implies immediately that the complement 
$\Gamma\backslash \tau$ consists of $L=b_1(\Gamma)$ edges. We need to show that if the
original matrices $T_i$ satisfy the momentum conservation condition $\sum_{s(e_i)=v} \bar T_i = \sum_{t(e_i)=v} \bar T_i$, 
then there is a non-empty set of $\epsilon$-deformations $\rho_\epsilon$ such that the associated matrices $T_{i,\epsilon}$
also satisfy the momentum conservation condition $\sum_{s(e_i)=v} \bar T_{i,\epsilon} = \sum_{t(e_i)=v} \bar T_{i,\epsilon}$. 
We will show that this is possible by first choosing a spanning tree $\tau$ for the Feynman graph $\Gamma$, 
then by constructing a $\epsilon$-deformation $q_{i,\epsilon}$ of the quadratic forms $q_i$ associated to the $L$
edges in the complement of the spanning tree, and finally by showing that there is a unique way to
extend the deformation to the remaining quadratic forms $q_i$ associated to the edges of the
spanning tree so that momentum conservation condition (as well as the above conditions (i)-(iii)) holds.

If $\tau$ is a spanning tree of $\Gamma$, then contracting $\tau$ to a single
vertex gives a bouquet of $L$ circles. Hence, the complement $\Gamma\backslash \tau$
provides a choice of $L$ edges all belonging  to different loops (=different generators of the
first homology) of $\Gamma$. This implies that, in the loop variables $\ell_i$, we can write the quadratic forms
$q_i$ with $i=1,\ldots, L$, associated to the edges $e_i$ in the complement $\Gamma\backslash \tau$
as functions of a single loop variable $\ell_i=(\ell_{i,1},\ldots, \ell_{i,D})$ and of the variable $x$,
independently of the remaining variables $\ell_j$ with $j\neq i$; see  Lemma~\ref{FQlem}. Now, we can deform $q_i$ into $q_{i,\epsilon}$ just by adding terms
of the form $\lambda_{j,k,\epsilon} \ell^2_{j,k}$ for the remaining variables that do not appear in $q_i$. Concretely, we have 
$q_{i,\epsilon}(x,\ell_1,\ldots, \ell_L) := q_i(\ell_i) + \sum_{j\neq i} \lambda_{j,k,\epsilon} \ell^2_{j,k}$, where the terms $\lambda_{j,k,\epsilon}$ are chosen so that the above conditions (i)-(ii)-(iii) are satisfied. It remains to show that it is possible to compatibly choose the deformations
$q_{i,\epsilon}$ of the remaining $q_i$, with $i=L+1,\ldots,n$, so that the momentum conservation condition $\sum_{s(e_i)=v} \bar T_{i,\epsilon} = \sum_{t(e_i)=v} \bar T_{i,\epsilon}$ as well as the above conditions (i)-(iii) also holds. These remaining edges are the edges
of the spanning tree $\tau$.

The strategy is to extend the deformation to the edges of the spanning tree, by imposing the
momentum conservation condition, starting with the ends of the tree and proceed inward.
Consider first the vertices that are incident to only one
edge in the spanning tree. This means that, for all other edges incident to the same vertex, the
deformation $q_{i,\epsilon}$ has already be assigned. Thus, the matrices $\bar T_{i,\epsilon}$
for all but one of the edges are knows, and the momentum conservation equation at the vertex
fixes what the matrix $\bar T_{i,\epsilon}$ for the last remaining edge (the one in the spanning tree)
should be.

Since the coefficient of $x^2$ is fixed to be the mass $m_i^2 >0$,
which we can leave undeformed, determining $\bar T_{i,\epsilon}$ suffices to determine 
the full $T_{i,\epsilon}$ for this remaining edge, hence the quadratic form $q_{i,\epsilon}$
is also determined.

To proceed to the next step, observe that, after this step there must be vertices
in the spanning tree for which the $q_{i,\epsilon}$ for all but one of the adjacent edges
have already been determined. Indeed, we can just remove from the graph all the vertices
for which all adjacent edges have $q_{i,\epsilon}$ already determined. The intersection of
the original spanning tree with the remaining graph is a spanning tree for this smaller
graph and we can repeat the first step.

This implies that we can again uniquely determine the $\bar T_{i,\epsilon}$ (hence the
$T_{i,\epsilon}$ and the $q_{i,\epsilon}$) for the remaining edge by imposing the
momentum conservation condition.
Iterating this procedure exhausts all the edges
of the spanning tree. Each time, in diagonal form the $T_{j,\epsilon}$ have non-zero
eigenvalues with either positive or negative sign, hence the corresponding $A_{j,\epsilon}=
T^\dagger_{j,\epsilon} T_{j,\epsilon}$ has strictly positive $\lambda_{j,\epsilon}>0$ and
satisfies (i)-(ii)-(iii) in addition to satisfying (iv) by construction.
\end{proof} 
\begin{definition}\label{FQmotdef}
Let $(\Gamma, m)$ be a Feynman graph equipped with a mass parameter $m$ (and with trivial external momentum $\kappa$). The associated {\em Feynman quadrics-motive} $M^Q_{(\Gamma,m)}$ is defined as the Voevodsky's mixed motive $M(\bbP^{LD}\backslash Q_{(\Gamma,m)})_\bbQ \in \mathrm{DM}_{\mathrm{gm}}(F)_\bbQ$, where 
$Q_{(\Gamma, m)}:=\bigcup_{i=1}^n Q_{i,\epsilon}$ stands for the union of the quadric hypersurfaces 
$Q_{i,\epsilon} \subset \bbP^{LD}$ introduced in Proposition \ref{3quaddef}.
\end{definition}

\begin{remark}\label{rk:last}
Let $(\Gamma, m, \kappa)$ be a Feynman graph equipped with a mass parameter and with external momentum. Note that Definition \ref{FQmotdef} holds similarly in the case where the ingoing momentum (in the left-hand side of \eqref{linrels}) equals the outgoing momentum (in the right-hand side of \eqref{linrels}). In particular, we can also consider the massive sunset graph equipped with (ingoing=outgoing) momentum. 
\end{remark}

\subsection{Regularization and Renormalization}\label{sub:Regularization}
In this subsection we express the divergent Feynman integral \eqref{IGamma} as a period of the Feynman quadrics-motive $M^Q_{(\Gamma, m)}$. Recall from \cite[\S8]{ItZu} that the {\em superficial degree of divergence} of the Feynman graph $\Gamma$ is defined as $\delta(\Gamma) := L D - 2n$. When $\delta(\Gamma)\geq 0$, the Feynman integral \eqref{IGamma} is divergent at infinity 
(=UV divergence). In the particular case where $\delta(\Gamma)=0$ the Feynman integral \eqref{IGamma} is moreover logarithmically divergent.

Following Notation~\ref{uvardef}, consider the following differential form  
\begin{equation}\label{omegaalpha}
 \omega := \sum_{i=0}^{LD} (-1)^i u_i \, du_1\wedge \cdots \wedge \widehat{du_i}\wedge \cdots \wedge du_{LD}
\end{equation} 
as well as the algebraic differential forms
\begin{eqnarray}\label{etaalpha}
\eta_\alpha :=\frac{\omega}{\prod_{i=1}^n q_i^\alpha} && \eta_{\alpha,\epsilon} :=\frac{\omega}{\prod_{i=1}^n q_{i,\epsilon}^\alpha}\,,
\end{eqnarray}
where $\alpha \in \bbN$ and $q_i$, resp. $q_{i, \epsilon}$, stands for the quadratic form corresponding to the quadric hypersurface $Q_i$, resp. $Q_{i, \epsilon}$, of Lemma \ref{FQlem}, resp. Proposition \ref{3quaddef} 
\begin{lemma}\label{etaFInt}
The algebraic differential forms \eqref{etaalpha} have the following properties:
\begin{itemize}
\item[(i)] When $\alpha=1$, the Feynman integral \eqref{IGamma} is equal to
\begin{equation}\label{etaU0int}
\frac{C}{(2\pi)^D} \cdot \int_{\bbA^{LD}(\bbR)} \eta_{1}\,,
\end{equation}
where $\bbA^{LD} \subset \bbP^{LD}$ stands for the affine chart of coordinates $(1,u_1,\ldots,u_{LD})$.
\item[(ii)] When $\alpha > \frac{LD}{2n}$, the following integral
\begin{equation}\label{intetaalpha}
\int_{\bbA^{LD}(\bbR)} \eta_{\alpha,\epsilon} = \int_{\bbP^{2D}(\bbR)} \eta_{\alpha,\epsilon} 
\end{equation}
is convergent and a period of $\bbP^{LD}\backslash Q_{(\Gamma,m)}$.
\end{itemize}
\end{lemma}
\proof The restriction of the differential form $\omega$ to the affine chart $\bbA^{LD}$ of coordinates $(1,u_1,\ldots,u_{LD})$ is the standard affine volume form $du_1\wedge\cdots\wedge du_{LD}$. In other words, it is the volume form $d^D\ell_1\cdots d^D\ell_L = \delta(H_\Gamma)\, d^Dk_1\, \cdots \, d^D k_n$ in the integrand of \eqref{IGamma}, with $\delta(H_\Gamma)$ the delta function. 
The denominator in the integrand of \eqref{IGamma} is the product
of the quadratic forms $q_j$ (as in $\eta_{\alpha=1}$) and the integration
in \eqref{IGamma} is over real variables. Hence, the locus of integration is the set of real 
points $\bbA^{LD}(\bbR)$ of the affine chart $\bbA^{LD}$. This shows item (i).

The presence of the exponent $\alpha$ in the denominator of \eqref{etaalpha} 
changes the superficial degree of convergence of the integral from $\delta(\Gamma) = L D - 2n$ to $\delta_\alpha(\Gamma)= DL - 2n \alpha $. When $DL -2 n\alpha <0$, the integral on the left-hand-side
of \eqref{intetaalpha} is convergent at infinity. Since we always assume that $m_i>0$ for every edge $e_i$, there are
no further divergences in the domain of integration $\bbA^{LN}(\bbR)$. The choice of the $\epsilon$-deformations 
$Q_{i,\epsilon}$ of Proposition~\ref{3quaddef} ensures that the differential form
$\eta_{\alpha,\epsilon}$ also has no poles on the hyperplane at infinity $\bbP^{LN}(\bbR)\backslash \bbA^{LN}(\bbR)=
\bbP^{LN-1}(\bbR)$; otherwise, $q_{i,\epsilon}$ would have zeros on the real locus $\bbP^{LN}(\bbR)$. This yields the equality \eqref{intetaalpha}. The proof of item (ii) follows now from the fact that \eqref{intetaalpha} is manifestly a period of $\bbP^{LD}\backslash Q_{(\Gamma,m)}$.
\endproof
Recall that the process of extraction of finite values from divergent Feynman integrals consists of
two main steps:
\begin{itemize}
\item[(i)] {\em Regularization}: the replacement of divergent Feynman integrals by 
meromorphic functions with poles at the exponents of divergence.
\item[(ii)]  {\em Renormalization}: a pole subtraction procedure on these meromorphic
functions performed consistently with the combinatorics of subgraphs and
quotient graphs (nested subdivergences).
\end{itemize}
A general procedure for carrying out these steps is provided by the
Connes--Kreimer formalism of algebraic renormalization \cite{CoKr}; consult also \cite[\S1]{CoMa}\cite[\S5]{Mar}. In our setting, regularization is
obtained by combination the $\epsilon$-deformation of quadric hypersurfaces with an Igusa zeta function. On the other hand, renormalization is as in the Connes--Kreimer setting. Consider the following Igusa zeta function
\begin{equation}\label{Igusa}
\cI(s)=\int_{\bbP^{LD}(\bbR)} \eta_{s,\epsilon} ,
\end{equation}
where the integer $\alpha\in \bbN$ has been replaced by a complex variable $s$.
\begin{proposition}\label{renorm}
The Igusa zeta function $\cI(s)$ has a Laurent series expansion
\begin{eqnarray}\label{LaurentIgusa}
 \cI(s) = \sum_{k\geq N} \gamma_k \, (s-\alpha)^k && \alpha \in \bbZ
\end{eqnarray}
for some $N\in \bbZ$, where the coefficients $\gamma_k$ are periods
of $(\bbP^{LD}\backslash Q_{(\Gamma,m)})\times \bbA^k$.
 \end{proposition}

\proof
Similarly to the proof of Lemma~\ref{etaFInt}, we observe that the integral defining $\cI(s)$ is convergent for $\Re(s)>\frac{LD}{2n}$.
By writing it as in the left-hand-side of \eqref{intetaalpha}, we can then use \cite[Cor.~4.7]{Bern} in order to extend $\cI(s)$ to a meromorphic function $\cI_\Gamma(s)$ on the entire complex plane. This extension satisfies the functional equation
\begin{equation}\label{FunEq}
\cI_\Gamma(s) = a_1(s)\, \cI(s+1) + \cdots + a_k(s)\, \cI(s+k)
\end{equation}
for some $k\in \bbN$, where the $a_i(s)$ are rational functions. For $\alpha>\frac{LD}{2n}$,
we write the Laurent expansion at $\alpha$ as follows:
$$ \cI_\Gamma(s)= \sum_{k\geq 0} \frac{(-1)^k}{k!} \int_{\bbP^{LD}(\bbR)} \eta_{\alpha,\epsilon} 
\cdot \log^k (\prod_{j=1}^n q_{j,\epsilon}) \,\,\, (s-\alpha)^k. $$
We then argue as in \cite{BeBro2}, using the identity
\begin{eqnarray*}
\log(f(u)) = \int_0^1 \theta(u,t) & \text{with} & \theta(u,t)= \frac{f(u)-1}{(f(u)-1)t+1} \, dt \,,
\end{eqnarray*}
to show that $\gamma_k$ is a period 
$$ \gamma_k = \frac{(-1)^k}{k!} \int_{\bbP^{LD}(\bbR) \times [0,1]^k} \eta_{\alpha,\epsilon} \wedge
\theta(u,t_1) \wedge \cdots \wedge \theta(u,t_k) $$
of $(\bbP^{LD}\backslash Q_{(\Gamma,m)})\times \bbA^k$. Using the
functional equation \eqref{FunEq}, it is then possible to argue inductively as in \cite{BeBro2}: the same property continues to hold for integers $\alpha\leq \frac{LD}{2n}$, using the
Laurent expansion of the terms on the right-hand-side of the functional equation to define
that of the left-hand-side.
\endproof
\begin{remark}
The coefficients $\gamma_k$ of the Laurent series expansion  \eqref{LaurentIgusa} are 
periods of $(\bbP^{2D}\backslash Q_{(\Gamma,m)})\times \bbA^k$. Therefore, thanks to $\bbA^1$-homotopy invariance, it suffices to understand the nature of the Feynman quadrics-motive $M^Q_{(\Gamma, m)}$. 
\end{remark}
We now briefly recall how the formalism of algebraic renormalization of \cite{CoKr}
can be used in order to obtain a renormalized value from the regularized Feynman
integral $\cI_\Gamma(s)$. Let $\cH_{CK}$ be the Hopf algebra of Feynman graphs. As a commutative
algebra, $\cH_{CK}$ is the commutative polynomial algebra in the connected and 1-edge connected
(1PI in the physics termonology) Feynman graphs. The coproduct $\Delta$ is non-cocommutative
and is defined by a certain sum over certain subgraphs (consult \cite[\S5.3 and \S6.2]{CoMa} for details):
$$ \Delta(\Gamma) = \Gamma\otimes 1 + 1 \otimes \Gamma + \sum_{\gamma \subset \Gamma} 
(\gamma \otimes \Gamma/\gamma)\,.$$ 
The Hopf algebra $\cH_{CK}$ is graded by the loop number $L$ (or by the number of internal
edges $n$) and is connected. The antipode is defined inductively by the formula 
$S(X)=-X+\sum S(X') X''$ for $\Delta(X)=X\otimes 1 + 1 \otimes X +\sum X'\otimes X''$
with $X'$ and $X''$ of lower degree. 
Let $\cR$ be the algebra of Laurent series centered at $s=1$ and let $\cT$ be
the projection onto its polar part. This is a Rota--Baxter operator of degree $-1$, \ie the following equality holds $\cT(f_1)\cT(f_2)=\cT(f_1 \cT(f_2))+\cT(\cT(f_1)f_2)-\cT(f_1 f_2)$. The
operator $\cT$ determines a splitting $\cR_+=(1-\cT)\cR$ and $\cR_-=\cT\cR^u$,
called the {\em unitization} of $\cT\cR$. 
Given any morphism of commutative algebras $\phi: \cH_{CK} \to \cR$,
the coproduct on $\cH$ and the Rota--Baxter operator on $\cR$ determine
a Birkhoff factorization of $\phi$ into algebra homomorphisms 
$\phi_\pm: \cH_{CK}\to \cR_\pm$. These algebra homomorphisms are determined inductively by the following formulas
\begin{eqnarray*}
 \phi_-(X) & = & -\cT(\phi(X) +\sum \phi_-(X') \phi(X'') ) \\
 \phi_+(X) & = & (1-\cT)(\phi(X) +\sum \phi_-(X') \phi(X'') )
 \end{eqnarray*}
for $\Delta(X)=X\otimes 1 + 1 \otimes X +\sum(X'\otimes X'')$ in such a way that $\phi=(\phi_-\circ S)\star \phi_+$, with 
$\phi_1\star \phi_s (X):=\langle \phi_1\otimes \phi_2,\Delta(X)\rangle$.
Given a Feynman graph $\Gamma$, the Laurent series $\phi_+(\Gamma)(s)$ 
is regular at $s=1$ and the value $\phi_+(\Gamma)(1)$ is the renormalized
value. More explicitly, we have the following equality
$$ \phi_+(\Gamma)(s) =(1-\cT)(\cI_\Gamma(s) +\sum_{\gamma\subset \Gamma}
\phi_-(\gamma)(s) \cdot \cI_{\Gamma/\gamma}(s))\,,$$
where $\phi_-$ is defined inductively as above. 
\begin{remark}
in this article, we focus solely on the leading term $(1-\cT)\cI_\Gamma(s)|_{s=1}$
of the renormalized $\phi_+(\Gamma)$, which is a period of 
$(\bbP^{2D}\backslash Q_{(\Gamma,m)})\times \bbA^1$.
\end{remark}
\section{Proof of Theorem \ref{thm:main-new}}\label{sec:proof}
By construction, the massive sunset graph $(\Gamma, m)$ has $2$ vertices and $3$ internal edges $e_1, e_2, e_3$. Following \S\ref{Qsec}, let us write $(k_1, k_2, k_3) \in \bbA^{3D}$ for the associated ``momentum variables''. Under these notations, the two linear relations \eqref{linrels} reduce to the single relation $k_1 + k_2 + k_3=0$. Therefore, $N=1$ and $n-N=2$. Let us now choose $\ell_1:=k_1$ and $\ell_2:=k_2$ as the ``loop variables''. Equivalently, let us use the variables $u=(u_0,u_1,\ldots,u_{2D})$ with $u_0=x$, $\ell_1=(\ell_{1,1},\ldots, \ell_{1,D})=(u_1,\ldots, u_D)$, and $\ell_2=(\ell_{2,1}, \ldots, \ell_{2,D})=(u_{D+1},\ldots, u_{2D})$. Under these choices, the quadric hypersurfaces $Q_1, Q_2, Q_3 \subset \bbP^{2D}$ of Lemma \ref{FQlem} can be written as follows: 
\begin{equation}\label{sunQim}
\begin{array}{lcl}
 Q_1= \{ q_1(u) =\langle u, A_1u \rangle =0 \}& \text{with} & A_1={\rm diag}(m_1^2,\underbrace{1,\ldots,1}_{D},\underbrace{0,\ldots,0}_{D}) \\[2mm]
 Q_2= \{ q_2(u)=\langle u, A_2 u \rangle =0 \} & \text{with} & A_2={\rm diag}(m_2^2,\underbrace{0,\ldots,0}_{D},
\underbrace{1,\ldots,1}_{D}) \\[2mm]
 Q_3= \{ q_3(u)= \langle u, A_3 u\rangle =0 \}& \text{with} & A_3={\rm diag}(m_3^2, \underbrace{1,\ldots,1}_{D}, \underbrace{1,\ldots,1}_{D})\,.
\end{array}
\end{equation} 
Let us write $Q_{1, \epsilon}, Q_{2, \epsilon}, Q_{3, \epsilon} \subset \bbP^{2D}$ for the associated $\epsilon$-deformations of Proposition \ref{3quaddef}. An explicit choice for these $\epsilon$-deformations is the following
\begin{equation*}\label{3qepsilon}
 \begin{array}{lcl}
Q_{1, \epsilon} = \{q_{1,\epsilon}(u) =\langle u, A_{1,\epsilon} u \rangle = 0\} & \text{with} & 
A_{1,\epsilon}={\rm diag}(m_1^2,\underbrace{1,\ldots,1}_{D},\underbrace{\epsilon^2,\ldots,\epsilon^{2D}}_{D}) \\[2mm]
Q_{2, \epsilon} = \{q_{2,\epsilon}(u)=\langle u, A_{2,\epsilon} u\rangle =0\} & \text{with} & 
A_{2,\epsilon}={\rm diag}(m_2^2,\underbrace{\epsilon^2,\ldots,\epsilon^{2D}}_{D},\underbrace{1,\ldots,1}_{D}) \\[2mm]
Q_{3, \epsilon} = \{q_{3,\epsilon}(u)=\langle u, A_{3,\epsilon}u \rangle=0\}\,, & &
\end{array}
\end{equation*}
where $A_{3,\epsilon}$ stands for the following diagonal matrix:
$$ A_{3,\epsilon}=
{\rm diag}(m_3^2,  \underbrace{(1+\epsilon)^2,\ldots,(1+\epsilon^D)^2}_{D}, \underbrace{(1+\epsilon)^2,\ldots,(1+\epsilon^D)^2}_{D})\,. $$

\begin{proposition}\label{prop:2}
There exists a Zariski open subset $W(m)\subset \bbA^1$ (which depends on the mass parameter $m=(m_i)$) such that for every $\epsilon \in W(m)$ the above $\epsilon$-deformations $Q_{1, \epsilon}, Q_{2, \epsilon}, Q_{3, \epsilon} \subset \bbP^{2D}$  have not only the properties (i)-(iv) of Proposition \ref{3quaddef} but also the following additional properties:
\begin{itemize}
\item[(v)] The intersection $Q_{i,\epsilon}\cap Q_{j,\epsilon}$, with $i \neq j \in \{1, 2, 3\}$, is a complete intersection.
\item[(vi)] The intersection $Q_{1,\epsilon}\cap Q_{2,\epsilon} \cap Q_{3,\epsilon}$ is
a complete intersection. 
\end{itemize}
\end{proposition}
\begin{proof}
As explained in \cite[Prop.~17.18]{Harris}, an intersection of $r$ transversally intersecting hypersurfaces in $\bbP^s$, with $r<s$, is always a complete intersection. Therefore, the proof will consist on showing that the quadric hypersurfaces 
$Q_{1,\epsilon}, Q_{2,\epsilon}, Q_{3,\epsilon} \subset \bbP^{2D}$ intersect transversely.
This condition can be checked in an affine chart. Concretely, the tangent space at a point $\hat u$ of the quadric hypersurface $Q_{i,\epsilon}=\{q_{i,\epsilon}(u)=0\}$ is defined by the equation $\langle \nabla q_{i,\epsilon}(\hat u), (u-\hat u)\rangle =0$ (see Notation \ref{not:pairing}), where $\nabla q_{i,\epsilon}$ stands for the gradient vector. 
The gradient vectors $\nabla q_{1,\epsilon}(u)$, $\nabla q_{2,\epsilon}(u)$, and $\nabla q_{3,\epsilon}(u)$, are given, respectively, by the following expressions: 
$$ \begin{array}{l} 
(2 m_1^2 u_0, 2(u_1,\ldots, u_D), 2(\epsilon^2 u_{D+1}, \ldots, \epsilon^{2D} u_{2D})) \\[2mm]
(2 m_2^2 u_0, 2(\epsilon^2 u_1, \ldots, \epsilon^{2D} u_D), 2(u_{D+1},\ldots, u_{2D}) \\[2mm]
(2 m_3^2 u_0, 2((1+\epsilon)^2 u_1, \ldots, (1+\epsilon^D)^2 u_D), 2((1+\epsilon)^2 u_{D+1}, \ldots, 
(1+\epsilon^D)^2 u_{2D})\, . \end{array} $$
Hence, in order to prove item (v) it suffices to show that for every point $u \in \bigcup_{i\neq j} (Q_{i,\epsilon}\cap Q_{j,\epsilon})\subset  \bbP^{2D}$ any two of the three gradient vectors are linearly independent. Note that the points of $\bbP^{LD}$ which lie in at least one of the quadric hypersurfaces $Q_{i,\epsilon}$ have at least two nonzero coordinates $u_i$ (if all but one of the $u_i$ are zero,
then by the equation $q_{i,\epsilon}(u)=0$ the last coordinate must also be zero,
which would not be a point in projective space). Thus, it is enough to check
that at all points of $\bbP^{2D}$ with at least two non-zero coordinates, the
vectors are linearly independent. This is equivalent to checking that the following $2\times 2$ matrices have non-zero determinant:
$$ \begin{pmatrix} 1 & m_1^2 \\ \epsilon^{2j} & m_2^2 \end{pmatrix} \quad  \begin{pmatrix}
\epsilon^{2k} & m_1^2 \\ 1 & m_2^2\end{pmatrix}$$
$$\begin{pmatrix}
1 & \epsilon^{2k} \\ \epsilon^{2j} & 1 \end{pmatrix} \quad j \neq k
$$
$$  
\begin{pmatrix}
1 & 1 \\ \epsilon^{2j} & \epsilon^{2k} \end{pmatrix} \quad \begin{pmatrix} \epsilon^{2j} & \epsilon^{2k} \\
1 & 1 \end{pmatrix}$$
$$
 \begin{pmatrix} 
1 & m_1^2 \\ (1+\epsilon^j)^2 & m_3^2 \end{pmatrix} \quad  \begin{pmatrix} 
1 & m_2^2 \\ (1+\epsilon^j)^2 & m_3^2 \end{pmatrix} \quad 
\begin{pmatrix} \epsilon^{2j} 
& m_1^2 \\ (1+\epsilon^j)^2 & m_3^2 \end{pmatrix} \quad \begin{pmatrix} \epsilon^{2j} 
& m_2^2 \\ (1+\epsilon^j)^2 & m_3^2 \end{pmatrix}
$$
$$
\begin{pmatrix} 1 & \epsilon^{2k} 
 \\ (1+\epsilon^j)^2 & (1+\epsilon^k)^2 \end{pmatrix}\quad j \neq k $$
 $$
 \begin{pmatrix} \epsilon^{2j} & \epsilon^{2k} 
 \\ (1+\epsilon^j)^2 & (1+\epsilon^k)^2 \end{pmatrix}\quad  \begin{pmatrix} 1 & 1
 \\ (1+\epsilon^j)^2 & (1+\epsilon^k)^2 \end{pmatrix}  \,.
$$ 
The locus where at least one of these determinants is equal to zero defines a polynomial equation
in $\epsilon$ that depends on the value of the masses $m=(m_i)$. Thus, the set of $\epsilon$'s 
where all the determinants are nonzero is the complement of the solutions of these polynomial
equations, hence a Zariski open set $U=U(m)$. The intersection of this open set $U(m)$ with the open set
of sufficiently small $\epsilon$'s for which conditions (i)-(iv) of Proposition \ref{3quaddef} hold is also a Zariski open set.

In the same vein, in order to prove (vi), it suffices to show that at every point $u \in Q_{1,\epsilon}\cap Q_{2,\epsilon}\cap Q_{3,\epsilon}\subset  \bbP^{2D}$ the three gradient vectors are linearly independent. This means that, at every such point, at least one $3\times 3$-minor of the matrix
formed by the three gradient vectors is non-zero. It is therefore sufficient to show that the 
following $3\times 3$ matrices have non-zero determinant:
 $$ \begin{pmatrix} 1 & \epsilon^{2k} & m_1^2 \\ \epsilon^{2j} & 1 & m_2^2 \\
(1+\epsilon^j)^2 & (1+\epsilon^k)^2 & m_3^2 
\end{pmatrix} $$
$$ \begin{pmatrix} 1 & 1 & m_1^2 \\ \epsilon^{2j} & \epsilon^{2k} & m_2^2 \\
(1+\epsilon^j)^2 & (1+\epsilon^k)^2 & m_3^2 
\end{pmatrix} $$
$$  \begin{pmatrix} \epsilon^{2j} & \epsilon^{2k} & m_1^2 \\  1 & 1 & m_2^2 \\
(1+\epsilon^j)^2 & (1+\epsilon^k)^2 & m_3^2 
\end{pmatrix} $$
$$
\begin{pmatrix} \epsilon^{2i} & \epsilon^{2j} & \epsilon^{2k}  \\  1 & 1 & 1 \\
(1+\epsilon^i)^2 & (1+\epsilon^j)^2 & (1+\epsilon^k)^2 
\end{pmatrix}\quad 
\begin{pmatrix}   1 & 1 & 1 \\ \epsilon^{2i} & \epsilon^{2j} & \epsilon^{2k}  \\
(1+\epsilon^i)^2 & (1+\epsilon^j)^2 & (1+\epsilon^k)^2 
\end{pmatrix} $$
$$ 
\begin{pmatrix}   1 & 1 & \epsilon^{2k}  \\ \epsilon^{2i} & \epsilon^{2j} & 1  \\
(1+\epsilon^i)^2 & (1+\epsilon^j)^2 & (1+\epsilon^k)^2 
\end{pmatrix} \quad
\begin{pmatrix}   1 & \epsilon^{2j} & \epsilon^{2k}  \\ \epsilon^{2i} & 1 & 1  \\
(1+\epsilon^i)^2 & (1+\epsilon^j)^2 & (1+\epsilon^k)^2 
\end{pmatrix}\,.
$$
Once again, for each choice of the masses $m=(m_i)$, there exists a Zariski open set $V=V(m)$
on which all these determinants are non-zero. The intersection of the open sets
$U(m)$ and $V(m)$ with the open set of sufficiently small $\epsilon$'s for which
conditions (i)-(iv) of Proposition \ref{3quaddef} hold is also a Zariski open set $W(m)$ (which depends on the mass parameter $m=(m_i)$). This concludes the proof.
\end{proof}

\begin{notation}\label{not:key}
\begin{itemize}
\item[(i)] Let $X$ be a $F$-scheme of finite type. Following Voevodsky \cite{Voev}, we will write $M(X)_\bbQ$, resp. $M^c(X)_\bbQ$, for the mixed motive, resp. mixed motive with compact support, associated to $X$. Recall from {\em loc. cit.} that whenever $X$ is proper, we have a canonical isomorphism $M^c(X)_\bbQ \simeq M(X)_\bbQ$.
\item[(ii)] Let $\mathrm{Chow}(F)_\bbQ$ be the Grothendieck's (additive) category of Chow motives; see \cite[\S4]{Andre}. Given a smooth projective $k$-scheme $X$, we will write $\mathfrak{h}(X)_\bbQ$ for the associated Chow motive; when $X=\mathrm{Spec}(k)$, we will write ${\bf 1}_\bbQ$ instead.
\end{itemize}
\end{notation}
\begin{remark}\label{rk:key}
As proved by Voevodsky in \cite[Prop.~2.1.4]{Voev} (consult also \cite[\S18.3]{Andre}), there exists a fully-faithful (contravariant) functor $\Phi\colon \mathrm{Chow}(F)_\bbQ \to \mathrm{DM}_{\mathrm{gm}}(F)_\bbQ$ such that $\Phi(\mathfrak{h}(X)_\bbQ)\simeq M(X)_\bbQ$ for every smooth projective $F$-scheme $X$. Moreover, the functor $\Phi$ sends the Lefschetz motive $\bbL$ to the Tate motive $\bbQ(1)[2]$.
\end{remark}

\subsection*{Proof of item (i)}
\begin{lemma}\label{lem:compact}
For every smooth $F$-scheme $X$, the motive $M(X)_\bbQ$ is mixed-Tate if and only if the motive $M^c(X)_\bbQ$ is mixed-Tate.
\end{lemma}
\begin{proof}
Without loss of generality, we can assume that $X$ is equidimensional; let $d$ be its dimension. As proved by Voevodsky in \cite[Thm.~4.3.7]{Voev}, the dual $M(X)^\vee$ of $M(X)$ is isomorphic to $M^c(X)_\bbQ(-d)[-2d]$. Hence, the proof follows from the fact that the category of mixed-Tate motives is stable under duals and Tate-twists.
\end{proof}
Recall from Definition \ref{FQmotdef} that the Feynman quadrics-motive $M^Q_{(\Gamma, m)}$ is defined as $M(\bbP^{2D}\backslash Q_{(\Gamma, m)})_\bbQ \in \mathrm{DM}_{\mathrm{gm}}(F)_\bbQ$, where $Q_{(\Gamma, m)}=Q_{1, \epsilon}\cap Q_{2, \epsilon} \cap Q_{3, \epsilon}$. Since $\bbP^{2D}\backslash Q_{(\Gamma, m)}$ is smooth, we hence conclude from Lemma \ref{lem:compact} that
\begin{equation}\label{equivalence0}
M^Q_{(\Gamma,m)}\,\,\text{mixed}\text{-}\text{Tate} \Leftrightarrow M^c(\bbP^{2D}\backslash Q_{(\Gamma, m)})_\bbQ\,\,\text{mixed}\text{-}\text{Tate}\,.
\end{equation}
Let $X$ be a $F$-scheme of finite type. Recall from \cite[Ex.~16.18]{Maz} that given a Zariski open cover $X=U \cup V$, we have an induced Mayer-Vietoris distinguished triangle:
\begin{equation}\label{eq:triangle-MV}
M^c(X)_\bbQ \too M^c(U)_\bbQ \oplus M^c(V)_\bbQ \too M^c(U \cap V)_\bbQ \too M^c(X)_\bbQ[1]\,. 
\end{equation}
In the same vein, given a Zariski closed subscheme $Z\subset X$ with open complement $U$, recall from \cite[Prop.~4.1.5]{Voev} that we have an induced Gysin distinguished triangle:
\begin{equation}\label{eq:triangle-Gysin}
M^c(Z)_\bbQ \too M^c(X)_\bbQ \too M^c(U)_\bbQ \too M^c(Z)_\bbQ[1]\,.
\end{equation}
\begin{lemma}\label{lem:new}
The motive $M^c(\bbP^{2D}\backslash Q_{i, \epsilon})_\bbQ$ is mixed-Tate for every $i \in \{1,2,3\}$.
\end{lemma}
\begin{proof}
Since $2D$ is even, the hypersurface quadric $Q_{i, \epsilon} \subset \bbP^{2D}$ is odd-dimensional. Consequently, we have the following motivic decomposition (see \cite[Rk.~2.1]{Bro}):
\begin{equation}\label{eq:decomp-quadric}
\mathfrak{h}(Q_{i, \epsilon})_\bbQ \simeq {\bf 1}_\bbQ \oplus \bbL \oplus \bbL^{\otimes 2} \oplus \cdots \oplus \bbL^{\otimes (2D-1)} \,.
\end{equation}
This implies that the motive $M^c(Q_{i, \epsilon})_\bbQ \simeq M(Q_{i, \epsilon})_\bbQ$ is mixed-Tate; see Remark \ref{rk:key}. Using the fact that $M^c(\bbP^{2D})_\bbQ \simeq M(\bbP^{2D})_\bbQ$ is mixed-Tate, we hence conclude from the general Gysin triangle \eqref{eq:triangle-Gysin} (with $X:=\bbP^{2D}$ and $Z:=Q_{i, \epsilon}$) that the motive $M^c(\bbP^{2D}\backslash Q_{i, \epsilon})_\bbQ$ is also mixed-Tate. 
\end{proof}
\begin{proposition}\label{prop:1}
Assume that the motive $M^c(\bbP^{2D}\backslash (Q_{i, \epsilon} \cap Q_{j, \epsilon}))_\bbQ$ is mixed-Tate for every $i \neq j \in \{1,2,3\}$. Under this assumption, the motive $M^c(\bbP^{2D}\backslash Q_{(\Gamma, m)})_\bbQ$ is mixed-Tate if and only if the motive $M^c(\bbP^{2D}\backslash (Q_{1, \epsilon}\cap Q_{2, \epsilon} \cap Q_{3, \epsilon}))_\bbQ$ is mixed-Tate. 
\end{proposition}
\begin{proof}
Let $U:=\bbP^{2D}\backslash (Q_{1, \epsilon} \cup Q_{2, \epsilon})$ and $V:=\bbP^{2D}\backslash Q_{3, \epsilon}$. Note that $U \cap V = \bbP^{2D}\backslash Q_{(\Gamma, m)}$. Thanks to Lemma \ref{lem:new}, the motive $M^c(V)_\bbQ$ is mixed-Tate. Moreover, since by assumption the motive $M^c(\bbP^{2D}\backslash (Q_{1, \epsilon} \cap Q_{2, \epsilon}))_\bbQ$ is mixed-Tate, Lemma \ref{lem:1} below implies that the motive $M^c(U)_\bbQ$ is also mixed-Tate. Therefore, we conclude from the general Mayer-Vietoris triangle \eqref{eq:triangle-MV} (with $X:=U \cup V$) that
\begin{equation}\label{equivalence1}
M^c(\bbP^{2D}\backslash Q_{(\Gamma, m)})_\bbQ\,\,\text{mixed}\text{-}\text{Tate} \Leftrightarrow M^c(U \cup V)_\bbQ\,\,\text{mixed}\text{-}\text{Tate}\,.
\end{equation}
Now, let $U_{13}:=\bbP^{2D}\backslash (Q_{1, \epsilon} \cap Q_{3, \epsilon})$ and $U_{23}:=\bbP^{2D}\backslash (Q_{2, \epsilon} \cap Q_{3, \epsilon})$. Note that
$$ U_{13}\cap U_{23}= \bbP^{2D}\backslash ((Q_{1, \epsilon} \cap Q_{3, \epsilon}) \cup (Q_{2, \epsilon} \cap Q_{3, \epsilon}))= \bbP^{2D}\backslash ((Q_{1, \epsilon} \cup Q_{2, \epsilon})\cap Q_{3, \epsilon})= U\cup V$$
and that $U_{13}\cup U_{23}= \bbP^{2D}\backslash (Q_{1, \epsilon} \cap Q_{2, \epsilon} \cap Q_{3, \epsilon})$. Therefore, since by assumption the motives $M^c(\bbP^{2D}\backslash (Q_{1, \epsilon} \cap Q_{3, \epsilon}))_\bbQ$ and $M^c(\bbP^{2D}\backslash (Q_{2, \epsilon} \cap Q_{3, \epsilon}))_\bbQ$ are mixed-Tate, we conclude from the general Mayer-Vietoris triangle \eqref{eq:triangle-MV} (with $X:=U_{13} \cup U_{23}$) that
\begin{equation}\label{equivalence2}
M^c(U\cup V)_\bbQ\,\,\text{mixed}\text{-}\text{Tate} \Leftrightarrow M^c(\bbP^{2D}\backslash (Q_{1, \epsilon} \cap Q_{2, \epsilon} \cap Q_{3, \epsilon}))_\bbQ\,\,\text{mixed}\text{-}\text{Tate}\,.
\end{equation}
The proof follows now from the combination of \eqref{equivalence1} with \eqref{equivalence2}.
\end{proof}

\begin{lemma}\label{lem:1}
The motive $M^c(\bbP^{2D}\backslash (Q_{i, \epsilon} \cap Q_{j, \epsilon}))_\bbQ$ is mixed-Tate if and only if the motive $M^c(\bbP^{2D}\backslash (Q_{i, \epsilon} \cup Q_{j, \epsilon}))_\bbQ$ is mixed-Tate.
\end{lemma}
\begin{proof}
Thanks to Lemma \ref{lem:new}, the proof follows from the general Mayer-Vietoris triangle \eqref{eq:triangle-MV} (with $X:=\bbP^{2D}\backslash (Q_{i, \epsilon} \cap Q_{j, \epsilon})$, $U:=\bbP^{2D}\backslash Q_{i, \epsilon}$ and $V:=\bbP^{2D}\backslash Q_{j, \epsilon}$); note that under these choices we have $U\cap V = \bbP^{2D}\backslash (Q_{i, \epsilon} \cup Q_{j, \epsilon})$.
\end{proof}

Thanks to Proposition \ref{prop:2}, the intersection $Q_{i, \epsilon} \cap Q_{j, \epsilon}$, with $i \neq j \in \{1,2,3\}$, is a complete intersection of two odd-dimensional quadrics. Therefore, as proved in \cite[Cor.~2.1]{BerTab}, the Chow motive $\mathfrak{h}(Q_{i, \epsilon} \cap Q_{j, \epsilon})_\bbQ$ admits the motivic decomposition:
\begin{equation}\label{eq:decomp-intersection}
{\bf 1}_\bbQ \oplus \bbL \oplus \bbL^{\otimes 2} \oplus \cdots \oplus \bbL^{\otimes (D-2)} \oplus (\bbL^{\otimes (D-1)})^{\oplus (2D+2)} \oplus \bbL^{\otimes D} \oplus \cdots \oplus \bbL^{\otimes (2D-2)}\,.
\end{equation}
This implies that the motive $M^c(Q_{i, \epsilon} \cap Q_{j, \epsilon})_\bbQ\simeq M(Q_{i, \epsilon} \cap Q_{j, \epsilon})_\bbQ$ is mixed-Tate; see Remark \ref{rk:key}. Using the fact that the motive $M^c(\bbP^{2D})_\bbQ\simeq M(\bbP^{2D})_\bbQ$ is mixed-Tate, we hence conclude from the general Gysin triangle \eqref{eq:triangle-Gysin} (with $X:=\bbP^{2D}$ and $Z:=Q_{i, \epsilon} \cap Q_{j, \epsilon}$) that the motive $M^c(\bbP^{2D}\backslash (Q_{i, \epsilon} \cap Q_{j, \epsilon}))_\bbQ$ is also mixed-Tate. Consequently, thanks to Proposition \ref{prop:1}, we obtain the (unconditional) equivalence:
\begin{equation*}
M^c(\bbP^{2D}\backslash Q_{(\Gamma, m)})_\bbQ\,\,\text{mixed}\text{-}\text{Tate} \Leftrightarrow M^c(\bbP^{2D}\backslash (Q_{1, \epsilon} \cap Q_{2, \epsilon} \cap Q_{3, \epsilon}))_\bbQ\,\,\text{mixed}\text{-}\text{Tate}\,.
\end{equation*}
Using the fact that the motive $M^c(\bbP^{2D})_\bbQ\simeq M(\bbP^{2D})_\bbQ$ is mixed-Tate, we conclude from the Gysin triangle \eqref{eq:triangle-Gysin} (with $X:=\bbP^{2D}$ and $Z:=Q_{1, \epsilon} \cap Q_{2, \epsilon} \cap Q_{3, \epsilon}$)~that 
\begin{equation*}
M^c(\bbP^{2D}\backslash (Q_{1, \epsilon} \cap Q_{2, \epsilon} \cap Q_{3, \epsilon}))_\bbQ\,\,\text{mixed}\text{-}\text{Tate} \Leftrightarrow M^c(Q_{1, \epsilon} \cap Q_{2, \epsilon} \cap Q_{3, \epsilon})_\bbQ\,\,\text{mixed}\text{-}\text{Tate}\,.
\end{equation*}
The combination of equivalence \eqref{equivalence0} with the preceding two equivalences leads then to the following equivalence:
\begin{equation}\label{eq:equivalence5}
M^Q_{(\Gamma, m)}\,\,\text{mixed}\text{-}\text{Tate} \Leftrightarrow M^c(Q_{1, \epsilon} \cap Q_{2, \epsilon} \cap Q_{3, \epsilon})_\bbQ\,\,\text{mixed}\text{-}\text{Tate}\,.
\end{equation}
Thanks to Proposition \ref{prop:2}, the intersection $Q_{1, \epsilon} \cap Q_{2, \epsilon} \cap Q_{3, \epsilon}$ is a complete intersection of three odd-dimensional quadrics. Therefore, as proved in \cite[Cor.~2.1]{BerTab}, the Chow motive $\mathfrak{h}(Q_{1, \epsilon} \cap Q_{2, \epsilon} \cap Q_{3, \epsilon})_\bbQ$ admits the following motivic decomposition
\begin{equation}\label{eq:decomp-last}
{\bf 1}_\bbQ \oplus \bbL \oplus \bbL^{\otimes 2} \oplus \cdots \oplus \bbL^{\otimes (2D-3)} \oplus (\mathfrak{h}^1(J_a^{D-2}(Q_{1, \epsilon}\cap Q_{2, \epsilon} \cap Q_{3, \epsilon}))_\bbQ \otimes \bbL^{\otimes (D-1)})\,,
\end{equation}
where $J_a^{D-2}(Q_{1, \epsilon}\cap Q_{2, \epsilon} \cap Q_{3, \epsilon})$ stands for the $(D-2)^{\mathrm{th}}$ intermediate algebraic Jacobian of $Q_{1, \epsilon}\cap Q_{2, \epsilon} \cap Q_{3, \epsilon}$. Moreover, as proved by Beauville in \cite[Thm.~6.3]{Beau}, the abelian variety $J^{D-2}_a(Q_{1, \epsilon}\cap Q_{2, \epsilon} \cap Q_{3, \epsilon})$ is isomorphic, as a principally polarized abelian variety, to the Prym variety $\mathrm{Prym}(\widetilde{C}/C)$ mentioned in \S\ref{sec:intro}. 

We now claim that the motive $M^c(Q_{1 \epsilon}\cap Q_{2, \epsilon} \cap Q_{3, \epsilon})_\bbQ \simeq M(Q_{1, \epsilon} \cap Q_{2, \epsilon} \cap Q_{3, \epsilon})_\bbQ$ is {\em not} mixed-Tate. Recall from Totaro \cite[Cor.~7.3]{Tot} that $M(Q_{1, \epsilon}\cap Q_{2, \epsilon} \cap Q_{3, \epsilon})_\bbQ$ is mixed-Tate if and only if the Chow motive $\mathfrak{h}(Q_{1, \epsilon}\cap Q_{2, \epsilon} \cap Q_{3, \epsilon})_\bbQ$ is a direct summand of a finite direct sum of Lefschetz motives. Given a Weil cohomology theory $H^\ast$ (see \cite[\S3.3-3.4]{Andre}), the odd cohomology of a direct summand of a finite direct sum of Lefschetz motives is always zero. In contrast, since $\mathrm{Prym}(\widetilde{C}/C)\not\simeq 0$ is an abelian variety, it follows from the above motivic decomposition \eqref{eq:decomp-last} that  
$$H^1(\mathfrak{h}(Q_{1, \epsilon} \cap Q_{2, \epsilon} \cap Q_{3, \epsilon})_\bbQ)=H^1(\mathfrak{h}^1(\mathrm{Prym}(\widetilde{C}/C))_\bbQ)=H^1(\mathrm{Prym}(\widetilde{C}/C))\neq 0\,.$$
This implies the preceding claim. Consequently, the proof of item (i) follows now automatically from the above equivalence \eqref{eq:equivalence5}. 

\subsection*{Proof of item (ii)}
Recall from above that we have the Mayer-Vietoris triangles
$$ M^c(U\cup V)_\bbQ \too M^c(U)_\bbQ \oplus M^c(V)_\bbQ \too M^c(\bbP\backslash Q_{(\Gamma, m)})_\bbQ \too M^c(U \cup V)_\bbQ[1] $$
$$ M^c(\bbP\backslash Q_{123})_\bbQ  \too M^c(U_{13})_\bbQ \oplus M^c(U_{23})_\bbQ \too M^c(U \cup V)_\bbQ  \too M^c(\bbP\backslash Q_{123})_\bbQ[1]\,,$$
where $\bbP:=\bbP^{2D}$, $U:=\bbP^{2D} \backslash(Q_{1, \epsilon} \cap Q_{2, \epsilon})$, $V:=\bbP^{2D}\backslash Q_{3, \epsilon}$, $U_{13}:= \bbP^{2D}\backslash(Q_{1, \epsilon} \cap Q_{3, \epsilon})$, $U_{23}:= \bbP^{2D}\backslash (Q_{2, \epsilon} \cap Q_{3, \epsilon})$, and $Q_{123}:= Q_{1, \epsilon} \cap Q_{2, \epsilon} \cap Q_{3, \epsilon}$.  

Recall also from above that we have the following Gysin triangles:
$$
M^c(Q_{1, \epsilon} \cup Q_{2, \epsilon})_\bbQ   \too M^c(\bbP)_\bbQ \too M^c(U)_\bbQ \too M^c(Q_{1, \epsilon} \cup Q_{2, \epsilon})_\bbQ[1]
$$
$$
M^c(Q_{3, \epsilon})_\bbQ  \too M^c(\bbP)_\bbQ \too M^c(V)_\bbQ  \too M^c(Q_{3, \epsilon})_\bbQ[1]
$$
$$
M^c(Q_{1, \epsilon} \cap Q_{3, \epsilon})_\bbQ  \too M^c(\bbP)_\bbQ \too  M^c(U_{13})_\bbQ \too M^c(Q_{1, \epsilon} \cap Q_{3, \epsilon})_\bbQ[1]
$$
$$
M^c(Q_{2, \epsilon} \cap Q_{3, \epsilon})_\bbQ  \to M^c(\bbP)_\bbQ \too M^c(U_{23})_\bbQ \too M^c(Q_{2, \epsilon} \cap Q_{3, \epsilon})_\bbQ[1]
$$
$$ M^c(Q_{123})_\bbQ \too M^c(\bbP)_\bbQ \too M^c(\bbP\backslash Q_{123})_\bbQ \too M^c(Q_{123})_\bbQ[1]\,.$$
In what follows, we make use of the direct sum of the $1^{\mathrm{st}}$ and $2^{\mathrm{nd}}$ Gysin triangles as well as of the direct sum of the $3^{\mathrm{rd}}$ and $4^{\mathrm{th}}$ Gysin triangles. Since $M^c(\bbP)_\bbQ \simeq M(\bbP)_\bbQ \simeq \bigoplus_{i=0}^{2D} \bbQ(i)[2i]$, we can then conclude from the above motivic computations \eqref{eq:decomp-quadric}, \eqref{eq:decomp-intersection} and \eqref{eq:decomp-last} (see Remark \ref{rk:key}) and from the isomorphism $J_a^{D-2}(Q_{123})=\mathrm{Prym}(\widetilde{C}/C)$ that the motive $M^c(\bbP\backslash Q_{(\Gamma, m)})_\bbQ$ belongs to the smallest subcategory of $\mathrm{DM}_{\mathrm{gm}}(F)_\bbQ$ which can be obtained from the set of motives $\{M(\mathrm{Prym}(\widetilde{C}/C))_\bbQ, \bbQ(0), \bbQ(1)\}$ by taking direct sums, shifts, summands, tensor products, and at most $5$ cones. Now, recall from \cite[Thm.~4.3.7]{Voev} that since $\bbP^{2D}\backslash Q_{(\Gamma, m)}$ is smooth and $2D$-dimensional, we have a canonical isomorphism between the dual $M(\bbP^{2D}\backslash Q_{(\Gamma, m)})_\bbQ^\vee$ of $M^Q_{(\Gamma, m)}$ and $M^c(\bbP^{2D}\backslash Q_{(\Gamma, m)})_\bbQ(-2D)[-4D]$. Using the fact that the duality functor $(-)^\vee$ preserves direct sums, shifts, summands, tensor products, and cones, we hence conclude that the Feynman quadrics-motive $M^Q_{(\Gamma, m)}$ belongs to the smallest subcategory of $\mathrm{DM}_{\mathrm{gm}}(F)_\bbQ$ which can be obtained from the set of motives $ \{M(\mathrm{Prym}(\widetilde{C}/C))^\vee_\bbQ, \bbQ(-1)^\vee, \bbQ(1)^\vee\}$ by taking direct sums, shifts, summands, tensor products, and at most $5$ cones. The proof follows now from the isomorphisms $\bbQ(i)^\vee \simeq \bbQ(-i)$ and 
$$ M(\mathrm{Prym}(\widetilde{C}/C))_\bbQ^\vee \simeq M(\mathrm{Prym}(\widetilde{C}/C))_\bbQ (-d)[-2d]\,,$$ 
where $d$ stands for the dimension of the Prym variety.

\end{document}

\end{proof}